\documentclass[11pt, a4paper]{amsart}
\usepackage{amssymb}
\usepackage[lining]{ebgaramond}

\usepackage{mathrsfs}

\usepackage[all]{xy}
\usepackage[headings]{fullpage}

\usepackage[colorlinks=true,linkcolor=blue, citecolor=blue, urlcolor=blue,%
pagebackref=true]{hyperref}
\makeatletter

\@addtoreset{equation}{section}
\def\theequation{\thesection.\@arabic \c@equation}

\def\@citecolor{blue}
\def\@linkcolor{blue}
\def\@urlcolor{blue}
\def\theenumi{\@alph\c@enumi}

\makeatother

\swapnumbers
\theoremstyle{plain}

\newtheorem{lemma}[equation]{Lemma}

\newtheorem{proposition}[equation]{Proposition}

\theoremstyle{definition}

\newtheorem{remark}[equation]{Remark}
\newenvironment{remarkbox}[1][]{%
    \begin{remark}[#1] \pushQED{\qed}}{\popQED \end{remark}}

\newtheorem{example}[equation]{Example}
\newenvironment{examplebox}[1][]{%
    \begin{example}[#1] \pushQED{\qed}}{\popQED \end{example}}

\newtheorem{definition}[equation]{Definition}

\newtheorem{notation}[equation]{Notation}

\newtheorem{discussion}[equation]{Discussion}
\newenvironment{discussionbox}[1][]{%
    \begin{discussion}[#1]\pushQED{\qed}}{\popQED \end{discussion}}

\newtheorem{observation}[equation]{Observation}
\newenvironment{observationbox}[1][]{%
    \begin{observation}[#1]\pushQED{\qed}}{\popQED \end{observation}}

\newtheorem{construction}[equation]{Construction}

    {\setcounter{step}{0}}{} \newcounter{step}

\newcommand{\calA}{\mathcal A}

\newcommand{\calB}{\mathcal B}

\newcommand{\scrD}{\mathscr D}

\newcommand{\bbF}{\mathbb F}

\newcommand{\calP}{\mathcal P}
\newcommand{\frakP}{{\mathfrak P}}
\newcommand{\frakp}{{\mathfrak p}}

\newcommand{\frakQ}{{\mathfrak Q}}
\newcommand{\frakq}{{\mathfrak q}}

\newcommand{\ints}{\mathbb{Z}}

\def\to{\longrightarrow}

\DeclareMathOperator{\Trace}{Tr}

\DeclareMathOperator{\height}{ht}

\DeclareMathOperator{\Aut}{Aut}

\DeclareMathOperator{\Spec}{Spec}

\newcommand{\define}[1]{\emph{#1}}
\newcommand{\minus}{\ensuremath{\smallsetminus}}

\DeclareMathOperator{\Frac}{Frac}

\DeclareMathOperator{\GL}{GL}

\newcommand{\noetherDiff}{\ensuremath{\scrD_N}}
\newcommand{\dedekindDiff}{\ensuremath{\scrD_D}}

\DeclareMathOperator{\Ramif}{Ramif}

\title{Ramification in modular invariant rings}
\author{Manoj Kummini}
\address{Chennai Mathematical Institute, Siruseri, Tamilnadu 603103. India}
\email{mkummini@cmi.ac.in}
\thanks{MK was partially supported by an Infosys Foundation fellowship.}

\author{Mandira Mondal}
\address{School of Mathematics and Statistics, University of Hyderabad,
Telangana 500046. India}
\email{mandira@uohyd.ac.in}
\thanks{MM thanks Department of Science and Technology, Government of India
for the DST INSPIRE grant DST/INSPIRE/04/2021/001549.}

\begin{document}
\begin{abstract}
Let $p$ be a prime number, $\Bbbk$ a field of characteristic $p$ 
and $G$ a finite $p$-group acting on a standard graded polynomial ring
$S = \Bbbk[x_1, \ldots, x_n]$ as degree-preserving $\Bbbk$-algebra
automorphisms.
Assume that $G$ is generated by pseudo-reflections.
In our earlier work
(\emph{J. Pure Appl. Algebra}, vol. 228, no. 12,  2024)
we introduced a composition series of $G$.
In this note, we study the height-one ramification for the invariant rings
at the consecutive stages of this composition series.
We prove a condition for the extension $S^{G}\subseteq S^{G'}$ to split in
terms of the Dedekind different $\dedekindDiff(S^{G'}/S^G)$.
We construct an example illustrating that
$\dedekindDiff(S^{G'}/S^G)$ need not have `expected' generators.
\end{abstract}

\maketitle

\section{Introduction}
\label{section:intro}

Let $\Bbbk$ be a field and $G\subseteq \GL_n(\Bbbk)$ be a finite group.
Then $G$ acts on the polynomial ring of  
$n$ variables $S = \Bbbk[x_1, \ldots, x_n]$ (with $\deg x_i =1$ for all
$i$) as graded $\Bbbk$-algebra
automorphisms. We denote by $S^G$ the \define{invariant ring}, i.e., the
subring of $S$ consisting of all polynomials invariant under the action of
$G$.
It is known that if $S^G$ is a polynomial ring, then $G$ is generated by
pseudo-reflections, i.e., elements $g \in G$ such that $\{l \in S_1 \mid gl
=l\}$ is an $(n-1)$-dimensional $\Bbbk$-vector space.
The converse is true if $|G|\in \Bbbk^\times$. Moreover,
$S^G$ is a polynomial ring if and only if $S$ is a free (equivalently, flat)
$S^G$-module.
See~\cite[\S 7.2]{BensonInvBook1993}.

Now assume that $\Bbbk$ is a field of characteristic $p>0$ and that 
$G\subseteq \GL_n(\Bbbk)$ is a finite $p$-group.
In this situation, the Shank-Wehlau-Broer conjecture
(R.~J.~Shank and 
D.~L.~Wehlau~\cite[Conjecture~1.1]{ShankWehlauTransferModular1999}, 
reformulated by
A.~Broer~\cite[Corollary~4, p.~14]{BroerDirectSummandProperty2005})
asserts that 
$S^G$ is a polynomial ring if the extension $S^G\subseteq S$ splits, i.e.,
there is an $S^G$-linear map $S \to S^G$ such that the composite 
$S^G\subseteq S \to S^G$ is the identity map of $S^G$; Broer also showed
$G$ is generated by pseudo-reflections if the extension splits.

In~\cite[\S 5]{KumminiMondalPolyInv2022}, we introduced a composition series of
$G$, when $G$ is generated by pseudo-reflections; see
Lemma~\ref{lemma:conjTrans} and Proposition~\ref{proposition:compser}
below. 
Let $G'$ be the penultimate group in such a composition series.
In this case $G'$ is a normal subgroup of $G$ of index $p$ generated by
pseudo-reflections; further, $G/G'\simeq \mathbb{Z}/p\mathbb{Z}$ 
is generated by the class of a pseudo-reflection.
In this paper, we study height-one ramification of the extension $S^{G}
\subseteq S^{G'}$.

Section~\ref{section:trans} is about the Dedekind different of the
extension $S^{G} \subseteq S^{G'}$. A surprising fact is that in this
modular situation, height-one ramification is entirely due to the
inseparability of the residue field extension, unlike the non-modular
situation (Remark~\ref{remarkbox:ramif1}).
In Proposition~\ref{proposition:eltWithNZTr}, we prove a criterion for the
extension $S^G \subseteq S^{G'}$ to split in terms of the Dedekind
different.

In Section~\ref{Section:TheoryofLargerBeta}, we restrict our attention to
certain special $G$. For such $G$, 
we are able to give an explicit description of the generator of the
Dedekind different ideal $\dedekindDiff(S^{G'}/S^G)$ in
Proposition~\ref{proposition:BetaSigmaGreater}.
Using these results, we build Example~\ref{example:main}, where the
Dedekind different ideal $\dedekindDiff(S^{G'}/S^G)$ does not have any
`expected' generators, even though $S^G$ is a direct summand of $S^{G'}$.
(See the examples in the beginning of Section~\ref{section:example} for
context.)

\subsection*{Acknowledgements}
The computer algebra systems~\cite{M2} and~\cite{Singular}
provided valuable assistance in studying examples.

\section{Inertia groups and the Dedekind different}
\label{section:decin}

We collect a few facts about decomposition and inertia groups 
from~\cite[\S 41]{NagataLocalRings1962}.
Let $R$ be a noetherian normal domain with field of fractions $K$. 
Let $L$ be a finite Galois extension of $K$ with automorphism group $G$
and $S$ the integral closure of $R$ in
$L$. Then $G$ acts on $S$, and $S^G = S \cap K = R$.
Let $\frakq$ be a prime ideal of $S$. Write $G_{\mathrm{dec}}$ for the
\define{decomposition group} of $\frakq$, i.e., 
$\{g \in G \mid g\frakq = \frakq \}$.
Write $G_{\mathrm{in}}$ for the \define{inertia group} of $\frakq$, i.e.
$\{g \in G \mid gs-s \in \frakq \;\text{for all}\; s \in S\}$.
Note that for every $g \in G_{\mathrm{in}}$ and for every $s \in \frakq$, $gs \in
\frakq$, so $G_{\mathrm{in}}$ is a subgroup of $G_{\mathrm{dec}}$. 
In fact, $G_{\mathrm{dec}}$ acts on
$S/\frakq$, and $G_{\mathrm{in}}$ is the intersection of the stabilizers of the
elements of $S/\frakq$, so $G_{\mathrm{in}}$ is a normal subgroup of $G_{\mathrm{dec}}$.
Write 
$\frakp := \frakq \cap R,
\frakq_{\mathrm{dec}}  := \frakq \cap S^{G_{\mathrm{dec}}}, 
\frakq_{\mathrm{in}}  := \frakq \cap S^{G_{\mathrm{in}}}.$
The inclusion of groups $G \supseteq G_{\mathrm{dec}} \supseteq G_{\mathrm{in}} \supseteq 0$
yields the inclusion of integrally closed rings and their fraction fields 
\[
\xymatrix{
R \ar@{^{(}->}[r] \ar@{^{(}->}[d] & S^{G_{\mathrm{dec}} } \ar@{^{(}->}[r] \ar@{^{(}->}[d] & S^{G_{\mathrm{in}} }
\ar@{^{(}->}[r] \ar@{^{(}->}[d] & S\ar@{^{(}->}[d] \\
K \ar@{^{(}->}[r] & L^{G_{\mathrm{dec}} } \ar@{^{(}->}[r] & L^{G_{\mathrm{in}} }
\ar@{^{(}->}[r] & L.
}
\]
Then 
$\frakq$ is the unique prime ideal over $\frakq_{\mathrm{dec}}$;
the natural map
$\kappa(\frakp)  \to \kappa(\frakq_{\mathrm{dec } })$ is an isomorphism;
$\frakq_{\mathrm{dec } }$ and $\frakq_{\mathrm{in } }$ is unramified over
$R$; and $G_{\mathrm{dec}}/G_{\mathrm{in}}$ is the automorphism group for
the extensions
$\kappa(\frakp)\to \kappa(\frakq)$ and $\kappa(\frakp)\to
\kappa(\frakq_{\mathrm{in}})$~\cite[Theorem~41.2]{NagataLocalRings1962}.

The following lemma, in this generality, might be known, but
we could not find a source; with the additional hypothesis that $\height
\frakq = 1$ (or, equivalently, that $R$ is a DVR), 
it is proved in various books on number theory.

\begin{lemma}
\label{lemma:UnrIner}
Let $\frakq \in \Spec S$. 
Then $\frakq$ is
unramified over $R$ if and only if the inertia group 
of $\frakq$ is trivial.
\end{lemma}

\begin{proof}
`If' is immediate from the above discussion, since $S^{G_{\mathrm{in}}}=S$ and
$\frakq_{\mathrm{in}} = \frakq$.
For `only if', assume that $G_{\mathrm{in}}$ is non-trivial. We want to show that
$\frakq$ is ramified over $S^{G_{\mathrm{in}}}$. For this, we may assume that $G =
G_{\mathrm{in}}$, i.e., $gs-s \in \frakq$ for all $s \in S$ and $g \in G$.
Therefore $\frakp = \frakq_{\mathrm{dec}} = \frakq_{\mathrm{in}}$.
Hence $\frakq$ is the unique prime ideal lying over $\frakp$.
Without loss of generality,
we may further assume that $\frakp S_\frakq = \frakq S_\frakq$.
Therefore $S \otimes_R \kappa(\frakp)  = \kappa(\frakq )$.
Hence the degree 
$[\kappa(\frakq) : \kappa(\frakp)] \geq [L : K ] \geq |G| > 1$.
On the other hand, 
$|\Aut_{\kappa(\frakp)}(\kappa(\frakq))| = 1$ by the above discussion.
Therefore $\kappa(\frakq) / \kappa(\frakp)$ is not a Galois extension.
We claim that it is normal. Assume the claim.
Therefore $\kappa(\frakq) / \kappa(\frakp)$ is inseparable.
Hence $\frakq$ is ramified over $R$.

Now to prove the claim, 
let $\bar s \in \kappa(\frakq )$ and $s \in S$ a representative in $S$.
Then $\prod_{g \in G} (X - gs )$ is a monic polynomial in $R[X]$ (since $R$
is integrally closed) for which $s$ is a root.
Pass to $\kappa(\frakp)$.
This gives a monic polynomial that $\bar s$ satisfies over $\kappa(\frakp)$.
This polynomial splits into linear factors; hence so does the minimal
polynomial of $\bar s$ over $\kappa(\frakp)$.
This proves the claim.
\end{proof}

\begin{notation}
Write $\Ramif_1(S/R) = \{\frakq\in \Spec(S) \mid \height \frakq = 1
\;\text{and}\; \frakq \;\text{is ramified over}\; R\}$.
Let $G'$ be a normal subgroup of $G$ and $A = S^{G'}$.
Let $\Ramif_1(S/A/R) = \{\mathfrak{Q} \in \Spec S \mid 
\height (\mathfrak{Q}) = 1 \;\text{and}\; \mathfrak{Q} \cap A \in
\Ramif_1(A/R)\}$.
\end{notation}

We now restrict the above discussion to invariant rings.
Hereafter $S = \Bbbk[x_1,\ldots,x_n]$, where $\Bbbk$ is a field of
characteristic $p>0$. Set $\deg x_i = 1$ for all $i$.
Let $G$ be a finite $p$-group acting as degree-preserving $\Bbbk$-algebra
automorphisms of $S$. 
A \define{pseudo-reflection} in $G$ is an element $\tau$ such that
$\{l \in S_1 \mid \tau l = l \}$
is an $(n-1)$-dimensional subspace of $S_1$.
This is equivalent to saying that for all $s \in S$, the $S$-ideal
generated by 
$\{ \tau s - s \mid s \in S\}$
is a prime ideal generated by a non-zero element of $S_1$.
In this section, if $\tau$ is a pseudo-reflection, we will take $l_\tau$ to
be any linear form that generates the above ideal; in the next section, we
will impose some conditions for the sake of convenience.
Write $\calP$ for the set of pseudo-reflections in $G$. 
As mentioned above, $G'$ is a normal subgroup of $G$. 
Write $R = S^G$ and $A = S^{G'}$,
(the invariant subrings), $L = \Frac(S)$, 
$K = L^G$ (which equals $\Frac(R)$) and 
$F = L^{G'}$ (which equals $\Frac(A)$).

Let $\dedekindDiff(S/A)$, $\dedekindDiff(S/R)$ and $\dedekindDiff(A/R)$ be
the Dedekind differents of these ring extensions. These are homogeneous
ideals of $S$, $S$ and $A$
respectively. Since $S$ and $A$ and $R$ are UFDs
\cite[3.8.1]{CampbellWehlauModularInvThy11}, they are principal ideals. Fix
homogeneous generators $\Delta_{S/A}$, $\Delta_{S/R}$ and $\Delta_{A/R}$
respectively for these ideals.

\begin{discussionbox}
\label{discussionbox:AB}
We collect a few facts about ramification
from~\cite{AuslBuchsRamifThy1959}.
Write $\noetherDiff =  \noetherDiff(A/R)$ for the Noether different 
(called `homological different' in~\cite{AuslBuchsRamifThy1959}).
Abbreviate $\dedekindDiff(A/R)$ by $\dedekindDiff$.
Then 
\begin{enumerate}

\item
For all $\frakq \in \Spec A$, $\frakq$ is ramified over $R$ if and only if
$\noetherDiff \subseteq \frakq$~\cite[Theorem~2.7]{AuslBuchsRamifThy1959}.

\item
$\noetherDiff \subseteq \dedekindDiff$
(\cite[Proposition~3.1]{AuslBuchsRamifThy1959});

\item
If $\dedekindDiff$ is a proper ideal, then its associated prime ideals
are of height $1$.

\item
$\dedekindDiff$ and $\noetherDiff$ have identical 
height-$1$ associated components.
(Follows from~\cite[Proposition~3.3]{AuslBuchsRamifThy1959}, applied to the
localizations of $R$ at height-one prime ideals.)
\qedhere
\end{enumerate}
\end{discussionbox}

\begin{proposition}
\label{proposition:relDDRamif}
Suppose that $l\in S_1$.
Then
$\Delta_{A/R}\in lS$ if and only if $lS\in \Ramif_1(S/A/R)$.
\end{proposition}

\begin{proof}
We have the following equivalences:
\[
\Delta_{A/R}\in lS \Leftrightarrow
\dedekindDiff \subseteq lS \cap A \Leftrightarrow
\noetherDiff \subseteq lS \cap A \Leftrightarrow
 lS \cap A\in \Ramif_1(A/R) \Leftrightarrow 
 lS \in \Ramif_1(S/A/R).
\]
(The second equivalence follows from Discussion~\ref{discussionbox:AB}.)
\end{proof}

\begin{remarkbox}
\label{remarkbox:Deltaexp}
We record some (well-known) facts for later use.
Let $l \in S_1$.
Then the exponents of 
$l$ in $\Delta_{S/R}$ 
and in $\Delta_{S/S^{G_{\mathrm{in}}}}$ are the same, where $G_{\mathrm{in}}$ is the 
inertia group of the prime ideal $lS$ inside
$G$~\cite[Proposition~3.10.3]{BensonInvBook1993}.  
Note that $G_{\mathrm{in}} = \langle \sigma \in \calP \mid l_\sigma S= lS\rangle$.
Further, $S^{G_{\mathrm{in}}}$ is a polynomial 
ring~\cite[Theorem~5, Appendix]{LandweberStongDepth1987}.
Denote the degrees of the generators of 
$S^{G_{\mathrm{in}}}$ by $d_1, \ldots, d_n$.
Then
\[
\Delta_{S/S^{G_{\mathrm{in}}}} = l^{\sum_{i=1}^n(d_i-1)}.
\qedhere
\]
\end{remarkbox}

\begin{proposition}
\label{proposition:ramLoc}
Adopt the above notation. Then
\begin{enumerate}

\item
\label{proposition:ramLoc:SR}
$\Ramif_1(S/R) = \{l_\tau S \mid \tau \in \calP \}$.

\item
\label{proposition:ramLoc:SAR}
$\Ramif_1(S/A/R) = \{l_\tau S \mid \tau \in \calP \minus G'\}$.

\item
\label{proposition:ramLoc:AR}
$\Ramif_1(A/R) = \{\frakq \cap A \mid \frakq \in \Ramif_1(S/A/R) \}$.
\end{enumerate}
\end{proposition}

\begin{proof}
\eqref{proposition:ramLoc:SR}:
Let $\frakq \in \Ramif_1(S/R)$.
Then the inertia group of $\frakq$ is non-trivial
(Lemma~\ref{lemma:UnrIner}), i.e.,
there exists $1\neq \tau\in G$ such that $\tau s-s\in \frakq$ for all
$s\in S$. 
Hence there exists $1 \leq i \leq n$ such that 
$\frakq=(\tau x_i-x_i)S$ and $\tau x_j-x_j\in \Bbbk \langle \tau
x_i-x_i\rangle$ for all $1 \leq j \leq n$. 
In other words, $\tau$ is a pseudo-reflection. 
Hence $\frakq = ( \tau x_i-x_i)S = l_\tau S$.
Conversely, for all $\tau \in \calP$, 
the inertia group of $\l_\tau S$ is non-trivial (e.g., 
$\tau$ belongs to the inertia group of $\l_\tau S$), 
so $\l_\tau S \in \Ramif_1(S/R)$.

\eqref{proposition:ramLoc:SAR}:
Let $\frakq \in \Ramif_1(S/A/R)$.
Note that $\Ramif_1(S/A/R)  \subseteq \Ramif_1(S/R)$.
Therefore there exists $\sigma \in \calP$ such that $\frakq = l_\sigma S$.
Moreover, $\Delta_{A/R} \in l_\sigma S$
(Proposition~\ref{proposition:relDDRamif}).
Hence the exponent of 
$l_\sigma$ in $\Delta_{S/R}$ 
is more than that in 
$\Delta_{S/A}$
(using Proposition~\ref{proposition:DiffsProducts}).
These exponents are, respectively, the degrees of $\Delta_{S/S^H}$ and 
$\Delta_{S/S^{H'}}$, where $H$ and $H'$, respectively, 
are the inertia groups of $\frakq$ inside $G$ and $G'$  
respectively;~see Remark~\ref{remarkbox:Deltaexp}.
Therefore $H \neq H'$, i.e., there exists $\tau \in \calP \minus G'$
such that $\frakq= l_\tau S$.

Conversely, let $\tau \in \calP \minus G'$. Then the inertia groups of
$l_\tau S$ inside $G$ and inside $G'$ are different, so $\Delta_{A/R} \in 
l_\tau S$ (using Proposition~\ref{proposition:DiffsProducts}), and,
consequently, $l_\tau S \cap A$ is ramified over $R$.

\eqref{proposition:ramLoc:AR}:
Follows from~\eqref{proposition:ramLoc:SAR}.
\end{proof}

\begin{proposition}
\label{proposition:DiffsProducts}
There exists a non-zero element $u\in \Bbbk$ such that $\Delta_{S/R} =
u\Delta_{A/R}\Delta_{S/A}$ in $S$. Without loss of generality, we may
assume $\Delta_{S/R} = \Delta_{A/R}\Delta_{S/A}$.
\end{proposition}

\begin{proof} 
This follows from the proof of \cite[Lemma~3.10.1]{BensonInvBook1993},
where see that the inverse different for $R\rightarrow S$ is
\[
\dedekindDiff(S/R)^{-1}=\frac{1}{\Delta_{A/R}}\frac{1}{\Delta_{S/A}}.
\]
Hence $\dedekindDiff(S/R)$ is the principal ideal generated by
$\Delta_{A/R}\Delta_{S/A}$.  
\end{proof}

\begin{definition}
We say that the extension $R \subseteq A$ \define{splits} (or, $R$ is a
\emph{direct summand} of $A$) if there exists an $R$-module map 
$f : A \to R$ such that $f(r) = r$ for all $r \in R$.
\end{definition}

\begin{proposition}%
[\protect{\cite[Theorem~1]{BroerDirectSummandProperty2005}}]
\label{proposition:splitIffTrSurj}
$\Trace_{F/K}(A) \subseteq \Delta_{A/R }R$.
The extension $R \subseteq A$ splits if and only if equality holds.
\end{proposition}

\section{Action by transvection groups}
\label{section:trans}

We now assume that $G$ is generated by pseudo-reflections.
Every pseudo-reflection in $G$ is a transvection, so
$\calP$ denotes the set of transvections in $G$.
Transvection groups are important in the context of the Shank-Wehlau-Broer 
conjecture, as we mentioned in Section~\ref{section:intro}.

We recall some facts about the action of $G$ on $S$.
Without loss of generality, we may assume that, for all $g \in G$,
\begin{equation}
\label{equation:uTriang}
gx_1 = x_1 \;\text{and}\;
gx_i-x_i\in \Bbbk\langle x_{1},\ldots, x_{i-1}\rangle
\;\text{for all}\;2 \leq i \leq n.
\end{equation}
This follows 
from~\cite[Proposition~4.0.2]{CampbellWehlauModularInvThy11}.
In~\cite[Section~5]{KumminiMondalPolyInv2022}, we defined a composition
series of $G$ as follows.
For $g\in G$, let
\[
\beta_g :=\max \{j \mid x_j \; \text{appears with a non-zero coefficient in}\; gx_i-x_i
\;\text{for some}\;i\}.
\]
Let $\beta_G :=\max \{\beta_g\mid g\in \mathcal{P}\}$.
Then
\begin{lemma}[\protect{\cite[Lemma~5.3]{KumminiMondalPolyInv2022}}]
\label{lemma:conjTrans}
Let $g, h \in \calP$.
\begin{enumerate}

\item 
$ghg^{-1} \in \calP$ and $\beta_{ghg^{-1}} = \beta_h$.

\item
If $\beta_g = \beta_h$ then $gh=hg$. 

\end{enumerate}
\end{lemma}

As a consequence, we get a composition series of $G$.
\begin{proposition}%
[\protect{\cite[Proposition~5.4]{KumminiMondalPolyInv2022}}]
\label{proposition:compser}
$G$ has a composition series
$0  = G_0 \subsetneq G_1 \subsetneq \cdots \subsetneq G_k = G$
such that for each $1 \leq i \leq k$, 
\begin{enumerate}

\item
$G_i$ is a transvection group,

\item
$G_{i}/G_{i-1}$ is isomorphic to $\ints/p\ints$ and is generated by the
residue class of a transvection, and

\item
$\beta_{G_i} \geq \beta_{G_{i-1}}$.
\end{enumerate}
\end{proposition}

For the rest of the paper, we standardize the notation $l_\tau$ (where
$\tau$ is a transvection) as follows: $l_\tau$ is the unique element of
$S_1$ such that it generates the prime ideal generated by 
$\{ \tau s - s \mid s \in S\}$ and 
$x_{\beta_\tau}$ appears with coefficient $1$ in it.

Now assume that $G'$ is a normal subgroup of $G$ generated by transvections
such that $\beta_\sigma \geq \beta_{G'}$ for all $\sigma \in \calP \minus
G'$. 
For example, $G' = G_i$ for some $1 \leq i < k$ in the above composition
series.
For $s\in S$, let
\[
\Pi_{G'}s := \prod_{s' \in G's} s'
\]
where $G's$ is the $G'$-orbit $\{gs \mid g \in G \}$ of $s$.
Note that if the stabilizer of $s$ in $G'$ is non-trivial, then 
$\Pi_{G'}s  \neq \prod_{g \in G'} gs$.

We now show that $\Delta_{A/R}$ is $G$-invariant. We start with a lemma.
\begin{lemma}
\label{lemma:minOverRamif}
Let $\frakq \in \Ramif_1(A/R)$.
Then $\frakq=Ar$ for some $r\in R$ which is product of linear terms in
$S_1$.
\end{lemma}

\begin{proof}
Let $\tilde \frakq\in \Spec(S)$ be lying over $\frakq$. 
Then $\tilde \frakq \in \Ramif_1(S/A/R)$.
By Proposition~\ref{proposition:ramLoc}, there exists $\tau \in \calP
\minus G'$ such that $\tilde \frakq = l_\tau S$.

We know that $\{gl_\tau S \mid g\in G'\}$ is the set of all prime ideals of
$S$ lying over $\frakq$.
Hence $\frakq = (\Pi_{G'} l_\tau) A$.
We now argue that $\Pi_{G'} l_\tau$ is $G$-invariant.
Let $\sigma \in \calP$; we want to show that 
$\sigma \Pi_{G'} l_\tau = \Pi_{G'} l_\tau$.
Without loss of generality, $\sigma \not \in G'$.
By the definition of $G'$, $\beta_\sigma \geq \beta_{G'}$, so 
$g l_\tau \in \Bbbk \langle x_1, \ldots, x_{\beta_\sigma-1} \rangle$, and,
therefore, $\sigma g l_\tau =  g l_\tau$
for all $g \in G'$.
\end{proof}

\begin{proposition}
\label{proposition:DedekindDiffAR}
$\Delta_{A/R}$ is $G$-invariant.
\end{proposition}

\begin{proof}
Let $\frakq_1, \ldots, \frakq_t$ be the minimal prime ideals over 
$\Delta_{A/R}$.
Then $\frakq_i \in \Ramif_1(A/R)$ for all $i$.
By Lemma~\ref{lemma:minOverRamif}, there exist
homogeneous $r_1, \ldots, r_t \in R$ such that $\frakq_i = Ar_i$ for every
$i$. 
By Discussion~\ref{discussionbox:AB}, $\dedekindDiff(A/R)$ an unmixed
height-$1$ ideal contained inside $\displaystyle \bigcap_{i=1}^t \frakq_i$.
Hence $\dedekindDiff(A/R) = Ar_1^{n_1 }\cdots r_t^{n_t }$ 
for some positive integers $n_i$.
We may take $\Delta_{A/R} = r_1^{n_1 }\cdots r_t^{n_t }$.
\end{proof}

\begin{remarkbox}
\label{remarkbox:ramif1}
As another corollary of Lemma~\ref{lemma:minOverRamif}, we get the
following. 
Let $\frakq \in \Ramif_1(A/R)$; write $\frakp = \frakq \cap R$.
Then $\frakq = \frakp A$.
Therefore $\kappa(\frakq)$ is an inseparable extension of
$\kappa(\frakp)$.
This is unlike the non-modular situation:
If $\frakQ \in \Ramif_1(S/R)$, then, with $\frakP := \frakQ \cap R$, the
degree $[\kappa(\frakQ) : \kappa(\frakP)]$ divides $|G|$
(see, e.g., \cite[\S I.7, Corollary]{SerreLocalFieldsGTM1979}),
so it is a unit in $\Bbbk$, and, consequently, the extension 
$\kappa(\frakQ) / \kappa(\frakP)$ is separable.
Therefore $\frakP S_\frakQ \neq \frakQ S_\frakQ$.
\end{remarkbox}

Hereafter we further assume that $G'= G_{k-1}$ in the notation of
Proposition~\ref{proposition:compser}. Hence $G'$ is generated by
transvections and is normal in $G$. Further, $G/G' \simeq \ints/p\ints$,
generated by a transvection $\sigma \in G \minus G'$. 

\begin{lemma} 
\label{lemma:minDegElt}
Let $a\in A\minus R$ be a homogeneous element of smallest degree. Then 
\begin{enumerate}

\item
\label{lemma:minDegElt:inR}
$\sigma a-a\in R$.

\item
\label{lemma:minDegElt:TrSmalli}
$\Trace_{F/K}(a^k ) = 0$ for every $0 \leq k < p-1$;

\item
\label{lemma:minDegElt:TrpMinusOne}
$\Trace_{F/K}(a^{p-1 }) = -(\sigma a - a)^{p-1}$.
\end{enumerate}
\end{lemma}

\begin{proof}
Write $r = \sigma a - a$.
\eqref{lemma:minDegElt:inR}:
Let $\tau\in G'$. 
Since $G'$ is a normal subgroup of $G$, 
$\tau(\sigma a)=\sigma a$,
i.e. $\sigma a\in A$. Hence $r\in A$.

Let $l_\sigma \in S_1 $ be a linear form that defines the reflecting
hyperplane of $\sigma$.
Then $l_\sigma$ divides $r$ in $S$.
Hence $\Pi_{G'}l_\sigma$ divides $r$ in $A$.
By minimality of degree of $a$, we must have that
$r/\Pi_{G'}l_\sigma\in R$.
This implies that $r\in R$, since $\Pi_{G'}l_\sigma\in R$.
This proves~\eqref{lemma:minDegElt:inR}.

Before proving 
\eqref{lemma:minDegElt:TrSmalli} and
\eqref{lemma:minDegElt:TrpMinusOne}, 
we note that, for all $k \geq 0$,
\begin{equation}
\begin{aligned}
\label{equation:eqnTrace1}
\sum_{i=0}^{p-1}(\sigma^i(a)-a)^{k} & =
\sum_{i=0}^{p-1}\sum_{j=0}^{k}(-1)^j \binom{k}{j} \sigma^i(a)^ja^{k-j}
= 
\sum_{j=0}^{k}(-1)^j{\binom{k }{j }}a^{k-j}\sum_{i=0}^{p-1} \sigma^i(a^j)
\\
& =\sum_{j=0}^{k}(-1)^j\binom{k}{j}a^{k-j} \Trace_{F/K }(a^j).
\end{aligned}
\end{equation}

\eqref{lemma:minDegElt:TrSmalli}: 
We proceed by induction on $k$. If $k=0$, the assertion is clear.
Let $0 < k < p-1$.
(The induction step is required only if $p>2$.)
Assume that the assertion is true for all $0 \leq j < k$.
Then it follows from~\eqref{equation:eqnTrace1} that
\[
\sum_{i=0}^{p-1}(\sigma^i(a)-a)^{k} 
=(-1)^k \Trace_{F/K }(a^k).
\]
It is easy to check that $\sigma^i a - a = ir$ for all $0 \leq i < p$.
Therefore 
$\{ga - a \mid g \in G \}=\{\sigma^i a - a \mid 0 \leq i < p\}$
is a one-dimensional $\bbF_p$-vector space.
We apply~\cite[Proposition~4.3]{ShankWehlauTransferModular1999}
to this vector space to see that
$\sum_{i=0}^{p-1}(\sigma^i(a)-a)^{k}=0$. 
This proves~\eqref{lemma:minDegElt:TrSmalli}.

\eqref{lemma:minDegElt:TrpMinusOne}:
We use~\eqref{equation:eqnTrace1}, \eqref{lemma:minDegElt:TrSmalli} 
and~\cite[Proposition~4.3]{ShankWehlauTransferModular1999}
to
see that
\[
\Trace_{F/K }(a^{p-1})=\sum_{i=0}^{p-1}(\sigma^i(a)-a)^{p-1}
= \prod_{i=1}^{p-1} ir.
\]
Now observe that $(p-1)!  = -1 \in \bbF_p$.
\end{proof}

Proposition~\ref{proposition:splitIffTrSurj} 
implies that $\deg \Delta_{A/R} \leq (p-1)\deg a$; we see show below that
equality characterizes the splitting of the extension $R \to A$.

\begin{proposition}
\label{proposition:eltWithNZTr}
Let $a \in A \minus R$ be a homogeneous element of smallest degree. 
Then the following are equivalent:
\begin{enumerate}

\item
\label{proposition:eltWithNZTr:degDelta}
$\deg \Delta_{A/R} = (p-1)\deg a$.

\item
\label{proposition:eltWithNZTr:Delta}
$\Delta_{A/R} = \lambda \Trace_{F/K}(a^{p-1})$ for some $\lambda \in
\Bbbk^\times$.

\item
\label{proposition:eltWithNZTr:splits}
$R$ is a direct summand of $A$.

\end{enumerate}
\end{proposition}

\begin{proof}

\eqref{proposition:eltWithNZTr:degDelta}
$\implies$
\eqref{proposition:eltWithNZTr:Delta}: Note that
$\Trace_{F/K}(a^{p-1})$ is a nonzero element 
of $\Delta_{A/R }R$ 
(Proposition~\ref{proposition:splitIffTrSurj} and
Lemma~\ref{lemma:minDegElt}).
By the hypothesis, 
$\Trace_{F/K}(a^{p-1})$  and $\Delta_{A/R }$
have the same degree, so such a $\lambda$ exists.

\eqref{proposition:eltWithNZTr:Delta}
$\implies$
\eqref{proposition:eltWithNZTr:degDelta}: Immediate.

\eqref{proposition:eltWithNZTr:Delta}
$\implies$
\eqref{proposition:eltWithNZTr:splits}:
It follows that $\Delta_{A/R}R \subseteq \Trace_{F/K}(A)$.
However, in general, $\Trace_{F/K}(A) \subseteq \Delta_{A/R}R$
(Proposition~\ref{proposition:splitIffTrSurj}).
Hence the trace map $\Trace_{F/K} : A \to \Delta_{A/R}R$ is surjective,
and, consequently,
$R$ is a direct summand of $A$ 
(again, Proposition~\ref{proposition:splitIffTrSurj}).

\eqref{proposition:eltWithNZTr:splits}
$\implies$
\eqref{proposition:eltWithNZTr:Delta}:
By~\cite[Theorem~4(i)]{BroerHypersurfaces2006}, there exists
$a_1 \in A \minus R$ such that $A = R[a_1] \simeq 
R[T]/(T^p-\delta T -r_0 )$ where 
$\frac{\delta}{\Delta_{A/R}} \in \Bbbk^\times$ and
$r_0 = \prod_{i=0}^{p-1} \sigma^i a_1$.
Without loss generality, $a_1$
is homogeneous and $\deg a_1 = \min \{ j \mid A_j \neq R_j \}$.
Write $r = (\sigma-1)a$ and $r_1 = (\sigma-1)a_1$.
By Lemma~\ref{lemma:minDegElt}, $r \in R$,
$r_1 \in R$, $\Trace_{F/K}(a^{p-1}) = r^{p-1}$
and $\Trace_{F/K}(a_1^{p-1}) = r_1^{p-1}$.
It follows from the proof of~\cite[Theorem~4(i)]{BroerHypersurfaces2006},
that 
\[
\frac{\Trace_{F/K}(a_1^{p-1})}{\Delta_{A/R}} \in \Bbbk^\times
\]
Hence $\deg \Trace_{F/K}(a^{p-1}) = \deg \Delta_{A/R}$; therefore 
\[
\frac{\Trace_{F/K}(a^{p-1})}{\Delta_{A/R}} \in \Bbbk^\times.
\qedhere
\]
\end{proof}

\section{A special class of groups}
\label{Section:TheoryofLargerBeta}

Throughout this section, we assume that $G$ is such that 
$\beta_{\sigma} > \beta_{G'}$. 
In Proposition~\ref{proposition:BetaSigmaGreater}, we prove two equivalent
expressions for $\Delta_{A/R}$.

\begin{lemma}[\protect{\cite[Lemma~6.3]{KumminiMondalPolyInv2022}}]
\label{lemma:sigmaFixes}
We may assume that $\sigma x_i = x_i$ for all $i\neq n$.
\end{lemma}

\begin{proposition}
\label{proposition:largerBetaTransvection}
Let $\tau \in G \minus G'$.
Then $\tau$ is a transvection if and only if
$\tau = \sigma^k g$ for some $1 \leq k \leq p-1$ and a transvection $g \in
G'$ that satisfies $gx_i = x_i$ for all $i \neq n$.
\end{proposition}

\begin{proof}
`If' is immediate, so we prove `only if'.
Since $G'$ is normal inside $G$, there exists 
$1 \leq k \leq p-1$ and $g \in G'$ that $\tau = \sigma^k g$.
We need to show that $gx_i =  x_i$ for all $i \neq n$.
It would then follow that $g$ is a transvection.

Note that
$\sigma^k g x_n = x_n + k (\sigma x_n -x_n) + (gx_n -x_n)$, 
since $gx_n -x_n \in 
\Bbbk \langle x_1, \ldots, x_{\beta_\sigma-1 } \rangle$.
For the same reason, considered along with the fact that 
$\sigma x_n -x_n \not \in 
\Bbbk \langle x_1, \ldots, x_{\beta_\sigma-1 } \rangle$,
we see that 
\begin{equation}
\label{equation:sigmakgxn}
\sigma^k g x_n -x_n \not \in 
\Bbbk \langle x_1, \ldots, x_{\beta_\sigma-1 } \rangle.
\end{equation}
Now let $i < n$.
Since $g x_i \in \Bbbk \langle x_1, \ldots, x_i \rangle$, we see that
$\sigma^k g x_i = g x_i$.
Hence 
\begin{equation}
\label{equation:sigmakgxi}
\sigma^k g x_i - x_i \in 
\Bbbk \langle x_1, \ldots, x_{\beta_\sigma-1 } \rangle
\end{equation}
Since $\sigma^kg$ is a transvection and 
$\sigma^k g x_n -x_n \neq 0$ (by~\eqref{equation:sigmakgxn})
there exists $\lambda \in \Bbbk$ such that 
\[
\sigma^k g x_i - x_i  = \lambda (\sigma^k g x_n - x_n).
\]
By~\eqref{equation:sigmakgxi}, $\lambda = 0$.
Hence $gx_i = \sigma^k gx_i = x_i$.
\end{proof}

\begin{notation}
Write $H$ for the subgroup of $G'$ generated by transvections $g \in G'$
that satisfy $gx_i = x_i$ for all $i \neq n$.
\end{notation}

\begin{proposition}
\label{proposition:BetaSigmaGreater}
We have the following expressions for $\Delta_{A/R}$:
\[
\Delta_{A/R} 
= 
\prod_{\substack{l \in S_1 \\ l = l_\tau \;\text{for some}\; \tau \in
\calP \minus G'}} l^{p-1}
=
((\sigma -1) ( \Pi_H x_n ))^{p-1}.
\]
\end{proposition}

\begin{proof}
We start with the observation that 
$\beta_\tau=\beta_\sigma$ for all $\tau\in  \mathcal{P} \minus G'$.
Indeed, it follows from the proof of
Proposition~\ref{proposition:largerBetaTransvection} that $\beta_\tau >
\beta_{G'}$, so $\beta_\tau=\beta_\sigma$.

We now prove the first equality.
From Propositions~\ref{proposition:relDDRamif}
and~\ref{proposition:ramLoc}, it follows that 
a linear form $l\in S$ divides $\Delta_{A/R}$ if and only if $l=l_\tau$ for
some pseudo-reflection $\tau \in \calP \minus G'$.
Therefore, for all $\tau \in \calP \minus G'$, we need 
to determine the exponent of $l_\tau$ in $\Delta_{A/R}$.

Let $G_{\mathrm{in}}$ be the inertia group of the prime ideal $l_\tau S$ for the action
of $G$ on $S$.
Note that $G_{\mathrm{in}} = \langle g\in \mathcal{P} \mid l_g=l_\tau\rangle$.
Then $G_{\mathrm{in}} \cap G' = \{1 \}$ by the observation at the beginning of this
proof;
this is the inertia group of $l_\tau S$ for the action of $G'$ on $S$.
Therefore the exponents of $l_\tau$ in $\Delta_{A/R}$ and 
in $\Delta_{S/R}$ are the same. 
The exponent of $l_\tau$ in $\Delta_{S/R}$ equals the exponent
of $l_\tau$ in
$\Delta_{S/S^{G_{\mathrm{in}}}}$, 
which equals the degree of $\Delta_{S/S^{G_{\mathrm{in}}}}$
(Remark~\ref{remarkbox:Deltaexp}).
Note that
\[
G_{\mathrm{in}} \simeq \frac{G_{\mathrm{in}}}{G_{\mathrm{in}}\cap G'} \simeq \frac{G_{\mathrm{in}}G'}{G'} = \frac{G}{G'} \simeq 
\ints/p\ints.
\]
(Since $G=\langle G', \sigma\rangle$, we see that $G_{\mathrm{in}}G' = G$.)
Thus $G_{\mathrm{in}}=\langle \tau\rangle$. 
Therefore $S^{G_{\mathrm{in}}}$ is a polynomial ring generated by $n-1$ elements of degree
$1$ and one element of degree $p$,
from which it follows that $\deg \Delta_{S/S^{G_{\mathrm{in}}}} = p-1$.

We now prove the second equality.
Note that 
$\displaystyle     W := \{gx_n - x_n \mid g \in H\}$
is a finite-dimensional $\bbF_p$-vector subspace of 
$\Bbbk[x_1, \ldots, x_{n-1}]$.
Therefore 
$\displaystyle \Pi_H x_n = \prod_{w \in W }(x_n + w)$
is a $p$-polynomial in $x_n$, a notion that comes up in the context of
Dickson
invariants; see~\cite{WilkersonPrimerDickson1983}
or~\cite[\S 8.1]{SmithInvBook1995}.
Therefore $\displaystyle (\sigma-1)(\Pi_H x_n) = 
\prod_{w \in W }((\sigma-1)x_n + w)$.

Let $\calA = \{ (\sigma-1)x_n + w \mid w \in W\}$ and 
$\calB = \{l_\tau \mid  \tau \in \calP \minus G'\}$.
For all $g \in H$, 
\begin{equation}
\label{equation:sigmagxn}
\sigma g x_n = \sigma (x_n + (g-1)x_n )  = x_n +
(\sigma-1 ) x_n + (g-1)x_n.
\end{equation}
Therefore $ \calA = \{ (\sigma g-1) x_n \mid g \in H\} $.
Let $\varphi : \calA \to \calB$ be the map that sends
$(\sigma g-1) x_n$  to $l_{\sigma g}$. 
(Note that $\sigma g \in \calP \minus G'$.)
This map is injective, which we see as follows.
Suppose that $l_{\sigma g_1 } = l_{\sigma g_2 }$
with $g_1,g_2 \in H$.
Then, by~\eqref{equation:sigmagxn}, 
there exist non-zero $\lambda_1, \lambda_2 \in \Bbbk$ such that 
\[
(\sigma-1 ) x_n + (g_i-1)x_n = (\sigma g_i-1)x_n 
= \lambda_i l_{\sigma g_i}, \qquad i=1,2
\]
Write $\lambda = \lambda_2/\lambda_1$.
Then
\[
(\sigma-1 ) x_n + (g_1-1)x_n = \lambda (\sigma-1 ) x_n + (g_2-1)x_n,
\]
from which it follows that
$(1-\lambda )(\sigma-1 ) x_n \in \Bbbk \langle x_1, \ldots,
x_{\beta_{\sigma}-1} \rangle$.
Hence $\lambda = 1$; again by~\eqref{equation:sigmagxn}, 
$(\sigma g_1-1)x_n = (\sigma g_2-1)x_n$.
Hence $\varphi$ is injective.

To see the surjectivity, let $\tau \in \calP \minus G'$.
Write $\tau = \sigma^i g$, with $1 \leq i \leq p-1$ and $g \in H$.
Let $1 \leq j \leq p-1$ be such that $ij \equiv 1 \mod p$.
Then $\tau =(\sigma g^j)^i$.
Therefore $l_\tau = l_{\sigma g^j} \in \varphi(\calA)$, as we saw above.
This proves the surjectivity.
Therefore the second equality follows up to multiplication by a non-zero
scalar.
\end{proof}

\begin{remarkbox}
\label{remarkbox:ppoly}
Let us record an argument in the above proof for future use.
The fact that 
$f(x_n) := \displaystyle \Pi_H x_n$ is a $p$-polynomial in $x_n$ 
is useful in computing $(\sigma-1)f(x_n)$.
Write $f(x_n) = x_nf_1 + x_n^p f_p + \cdots$, where $f_1, f_p, \ldots$ are
in $\Bbbk[x_1, \ldots, x_{n-1}]$.
Then 
$(\sigma-1)f(x_n)
= (\sigma-1)x_nf_1 + ((\sigma-1)x_n)^p f_p   + \cdots
$.
\end{remarkbox}

The next proposition and the paragraph after that give some information
about $H$.

\begin{proposition}
\label{proposition:commut}
Let $g \in G'$.

\begin{enumerate}

\item
\label{proposition:commut:H}
Then $\sigma g \sigma^{-1}g^{-1} \in H$.

\item
\label{proposition:commut:stab}
$\sigma g = g \sigma$ if and only if $g l_\sigma = l_\sigma$.

\end{enumerate}

\end{proposition}

\begin{proof}
\eqref{proposition:commut:H}:
We may assume that $\sigma g  \neq g \sigma$.
For every transvection $\tau \in G'$,
$\beta_{\sigma \tau \sigma^{-1} } = \beta_\tau <
\beta_\sigma$, so $\beta_{\sigma \tau \sigma^{-1} } \in G'$.
Expressing $g$ as a product of transvections in $G'$, we see that
$\beta_{\sigma g \sigma^{-1} } \in G'$,
and, hence, that
$\sigma g \sigma^{-1}g^{-1} \in G'$.
Let $i<n$.
Then $g^{-1} x_i \in \Bbbk \langle x_1, \ldots, x_{n-1 } \rangle$, 
so $\sigma^{-1}g^{-1} x_i = g^{-1} x_i$. 
Therefore $\sigma g \sigma^{-1}g^{-1} x_i = \sigma x_i = x_i$,
i.e., $\sigma g \sigma^{-1}g^{-1} \in H$.

\eqref{proposition:commut:stab}:
Note that for all $i<n$, $g \sigma x_i = g x_i = \sigma g x_i$.
Therefore $\sigma g = g \sigma$ if and only if
$\sigma g x_n = g \sigma x_n$.
Write $l = (g-1)x_n$.
Now, $\sigma g x_n = \sigma(x_n + l) = x_n + l_\sigma + l$,
since $l \in \Bbbk \langle x_1, \ldots, x_{n-1} \rangle$.
On the other hand,
$g \sigma x_n = g(x_n + l_\sigma) =  x_n + l + l_\sigma +
(g-1) l_\sigma$.
Therefore $\sigma g = g \sigma$ 
if and only if $g l_\sigma =  l_\sigma$.
\end{proof}

Write $G''$ for the stabilizer of $l_\sigma$ inside $G'$.
Note that $H \subseteq G''$.
Using Proposition~\ref{proposition:commut} we can see that
the map $G' \to H$, $g \mapsto \sigma g \sigma^{-1} g^{-1 }$
induces an injective map $G'/G'' \to H$.
It need not be surjective: e.g., if $G'$ is abelian, then the map $G' \to
H$ is the constant map $g \mapsto 1$ but $H$ could be non-trivial; see
Example~\ref{example:main}, $H = \langle \tau_3 \rangle$.

\section{An example}
\label{section:example}

In this section we study some examples in the situation considered above where $R$ is a direct summand of $A$. We study the generating element $a$ of $A$ as an algebra over $R$. It may appear, as shown in the examples below, 
$a=\Pi_{G'} s$ for some $s\in S_1$, not fixed by $\sigma$. However, it will
be an error to conclude this to be true in general, as we shall show in Example \ref{example:main}.
 
\begin{examplebox}[\protect{\cite[Example~4.7]{ShankWehlauTransferModular1999}}]
Let $S = \bbF_2[x_1, \ldots, x_4]$.
Consider the $\bbF_2$-algebra automorphisms of $S$ given by
\[
\tau :
\begin{cases}
x_1 \mapsto x_1\\
x_2 \mapsto x_2 + x_1\\
x_3 \mapsto x_3\\
x_4 \mapsto  x_4
\end{cases},
\;
\sigma :
\begin{cases}
x_1 \mapsto x_1\\
x_2 \mapsto x_2 \\
x_3 \mapsto x_3\\
x_4 \mapsto x_4 + x_3
\end{cases}.
\]
Let $G = \langle \tau, \sigma \rangle$.
Let $G' =\langle\tau\rangle$. Then $G=\langle G', \sigma\rangle$. 

Direct calculation shows that
$\Bbbk[x_1,x_2^2+x_2x_1,x_3, x_4^2+x_4x_3] \subseteq R$.
Since $S/(x_1,x_2^2+x_2x_1,x_3, x_4^2+x_4x_3)$ is a four-dimensional
$\bbF_2$-vector-space, we see that
$\Bbbk[x_1,x_2^2+x_2x_1,x_3, x_4^2+x_4x_3] = R$.
Similarly $A = \Bbbk[x_1,x_2^2+x_2x_1,x_3, x_4]$.
Both are polynomial rings; therefore $R$ is a direct summand of $A$.
Since $A_1 \neq R_1$, we know, by
Proposition~\ref{proposition:eltWithNZTr}, that
$(1+\sigma)(a)$ generates the Dedekind different $\dedekindDiff(A/R)$
for every $a \in A_1 \minus R_1$. 
We may take $a = x_4$ (in which case $(1+\sigma)(a)=x_3$). 
Note that $a=\Pi_{G'} x_4$.

Let $H=\langle\sigma\tau\rangle$. 
Even though $H$ is not a transvection group, we can do a similar
calculation. 
Write $B = S^H$.
Then $B=R[x_4x_1+x_3x_2]$, so $B_1 = R_1$, $B_2 \neq R_2$
and $\Delta_{B/R} = (1+\sigma)(x_4x_1+x_3x_2) = x_1x_3$. 
Now take $b=\prod_{H}(x_4+x_2)=x_4^2+x_4x_3+x_2^2+x_2x_1+x_4x_1+x_3x_2$.
Then $(1+\sigma)b = \Delta_{B/R}$.
\end{examplebox}

\begin{examplebox}[Stong\protect{\cite[\S 8.1]{CampbellWehlauModularInvThy11}}]
Let $\Bbbk = \mathbb{F}_{p^3}$, with $\mathbb{F}_p$-basis $\{1, \omega, \mu \}$.
Let $\rho, \tau, \sigma \in \GL_3(\Bbbk )$, whose action on
$S = \Bbbk[x, y, z]$ is given by
\[
\rho :
\begin{cases}
x \mapsto x\\
y \mapsto y + x\\
z \mapsto z\\
\end{cases},
\;
\tau :
\begin{cases}
x \mapsto x\\
y \mapsto y\\
z \mapsto z + x\\
\end{cases},
\;
\sigma :
\begin{cases}
x \mapsto x\\
y \mapsto y + \omega x \\
z \mapsto z + \mu x\\
\end{cases}.
\]
Let $G' = \langle \rho, \tau \rangle$ and 
$G = \langle \rho, \tau, \sigma \rangle$.
Then we have the following:
\begin{align*}
A & = \Bbbk[x, N_2, N_3] \;\text{where}\; N_2=y^p-yx^{p-1},
N_3=z^p-zx^{p-1};
\\
R & = \Bbbk[x, (\omega^p-\omega)N_2-(\mu^p-\mu)N_3, 
N_2^p-(\omega^p-\omega)^{p-1}N_2x^{p(p-1)}].
\end{align*}
Since $R$ and $A$ are polynomial rings, $R$ is a direct summand of $A$.
Note that $p = \min \{j \mid A_j \neq R_j \}$.
We can take $a$ to be $N_2=\Pi_{G'}y$, or $N_3=\Pi_{G'}z$.

We could also consider $G' = \langle \rho, \sigma \rangle$ or 
$G' = \langle \tau, \sigma \rangle$. In these cases also, one can check
that there is a choice of $a$ that of the form $\Pi_{G'}s$ for some $s \in
S_1$.
\end{examplebox}

Now we construct the example mentioned in Section~\ref{section:intro}. 
Here $G$ and $G'$ are transvection groups such that $G'$ is a normal
subgroup of $G$ and $G=\langle G', \sigma\rangle$ with $\sigma$ a
transvection and $\beta_\sigma > \beta_{G'}$.

\begin{observationbox}
\label{observationbox:PiGprimes}
Let $s \in S$. Note that 
$\sigma(\Pi_{G'}s) = \Pi_{G'}(\sigma s)$ since $G'$ is normal in $G$.
In particular, if $s \in S^{\langle \sigma\rangle}$, then 
$\Pi_{G'}s \in R$. On the other hand, 
assume that $\beta_\sigma > \beta_{G'}$ and that
$s\in S_1\minus S^{\langle \sigma\rangle}$.
Then $\prod_{G'}s \not \in R$.
It can be seen as follows: by
way of contradiction, if $\sigma(\prod_{G'}s)=\prod_{G'}s$, then
$\sigma(g_1s)=g_2s$ for some $g_1, g_2\in G'$. 
As we saw in the proof of Proposition~\ref{proposition:BetaSigmaGreater},
$(\sigma g_1)s-s = \sigma s-s + g_1 s -s$, so 
$\sigma(s)-s = g_2s-s-(g_1s-s)$, which is a contradiction since
$\beta_{\sigma} > \beta_{G'}$.
In particular, $\deg (\prod_{G'}s) \geq \min \{j \mid A_j \neq R_j\}$.
In the example below, we show that equality does not hold for any 
$s\in S_1\minus S^{\langle \sigma\rangle}$.
\end{observationbox}

\begin{example}
\label{example:main}
Let $\Bbbk$ be a field of characteristic $p>0$ containing two elements
$\alpha$, $\beta$ such that $\alpha \not \in \bbF_p$ and $\beta \not \in 
\bbF_p(\alpha )$. 
Let $S = \Bbbk[x_1, x_2, x_3, y ]$. Let $G=\langle \tau_1, \tau_2, \tau_3, \sigma\rangle$ with the action of $G$ on $S$ is defined as follows:
\[
\tau_1 :
\begin{cases}
x_1 \mapsto x_1\\
x_2 \mapsto x_2 + x_1\\
x_3 \mapsto x_3\\
y \mapsto y + x_1
\end{cases},
\;
\tau_2 :
\begin{cases}
x_1 \mapsto x_1\\
x_2 \mapsto x_2 + \alpha x_1\\
x_3 \mapsto x_3\\
y \mapsto y + x_1
\end{cases},
\;
\tau_3 :
\begin{cases}
x_1 \mapsto x_1\\
x_2 \mapsto x_2 \\
x_3 \mapsto x_3\\
y \mapsto y + \beta x_1
\end{cases},
\;
\sigma :
\begin{cases}
x_1 \mapsto x_1\\
x_2 \mapsto x_2 \\
x_3 \mapsto x_3\\
y \mapsto y + x_3
\end{cases}.
\]
$G$ is an abelian group. 
Let $G' = \langle \tau_1, \tau_2, \tau_3\rangle$. 
Then $G = \langle G', \sigma\rangle$.
Note that $\beta_{G'} = 1$ and $\beta_{G} = \beta_\sigma = 3$.
In the next propositions, we show that $R$ is a direct summand of $A$ 
but there does not exist $s \in S_1$ such that 
$\Delta_{A/R}  = \Trace_{F/K}((\Pi_{G'}s)^{p-1}) \in A$.
\end{example}

\begin{proposition}
\label{proposition:deltaARmainEx}
$\Delta_{A/R} = (x_3(x_3 + \beta x_1)\cdots (x_3 + (p-1)\beta x_1))^{p-1}$.
\end{proposition}

\begin{proof}
We apply the second expression for $\Delta_{A/R}$ from
Proposition~\ref{proposition:BetaSigmaGreater}.
In the notation of that proposition, $H = \langle \tau_3 \rangle$.
(We need to use that $\alpha \not \in \bbF_p$.)
Now use Remark~\ref{remarkbox:ppoly}.
\end{proof}

In the next proposition we find an element of $A_p \minus R_p$, which will
be used to prove that $R$ is a direct summand of $A$.

\begin{proposition}
\label{proposition:ApNeqRp}
Let $\lambda  = \frac{(\beta^{p-1}-1)} {(\alpha + \cdots + \alpha^{p-1})}$,
$\mu = (1 + \alpha + \cdots + \alpha^{p-1})\lambda$ 
and 
\[
a  = y^p - \beta^{p-1}x_1^{p-1}y - \lambda x_2^p + \mu x_1^{p-1}x_2.
\]
\begin{enumerate}
\item 
\label{proposition:ApNeqRp:aNotInR}
Then $a \in A \minus R$.

\item 
\label{proposition:ApNeqRp:sMinusa}
Let $s\in A\minus R$ be homogeneous of degree $p$.
Then $s-a\in\Bbbk[x_1,x_3]\subseteq R$.
\end{enumerate}
\end{proposition}

\begin{proof} 
\eqref{proposition:ApNeqRp:aNotInR}:
It is easy to see that $\tau_i(a)=a$ for $i=1,2,3$ and $\sigma(a) - a = -\beta^{p-1} x_1^{p-1}x_3  + x_3^p  \neq 0.$
Hence $a \in A \minus R$.

\eqref{proposition:ApNeqRp:sMinusa}:
Let 
$s = \sum_{i=0}^p s_iy^i \in A,$ with $s_i \in \Bbbk[x_1, x_2, x_3]$,
homogeneous of degree~$p-i$.
Then $s_i \neq 0$ for some $i>0$.
We will show that $s-a\in R$.

Let $i$ be the largest integer such that $1 \leq i < p$ and $s_i \neq 0$.
Then 
$\Delta_{y}^{(i)}(s) = s_i \in \Bbbk[x_1, x_2, x_3]^{G'} = 
\Bbbk[x_1, \Pi_{G'}x_2, x_3]$. Since $\deg s_i < p$, we further see that
$s_i \in \Bbbk[x_1, x_3]$.
Suppose that $i > 1$. Then
$\Delta_{y}^{(i-1)}(s) = s_{i-1} + i! s_iy \in A$.
Hence $$s_{i-1} + i! s_iy  = \tau_3(s_{i-1} + i! s_iy) = s_{i-1} + i! s_i (y +
\beta x_1),$$ a contradiction. Therefore 
$s = s_0 + s_1y + s_p y^p;$ with $s_1 \in \Bbbk[x_1, x_3].$ Moreover, $s_p\neq 0$. For if $s_p = 0$, then $\tau_3(s) = s_0 + s_1(y+\beta x_1)\neq s$. Therefore, without loss of generality, we assume $s_p = 1$.
We can now write
\begin{equation*}
\label{equation:exprOfsLeadingCoeffNormalised}
s = y^p + s_1y + s_0.
\end{equation*}
We now explicitly compute $s_1$. Note that
\[
\tau_3(s_0) + s_1(y+\beta x_1 ) + (y^p+\beta^px_1^p )
= \tau_3(s) = s = s_0 + s_1y + y^p.
\]
Since $\tau_3(s_0)=s_0$, we must have $ \beta s_1x_1 + \beta^p x_1^p  = 0, 
\;\text{i.e.,}\; s_1 = -(\beta x_1)^{p-1}.
$
Thus
\begin{equation}
\label{equation:exprOfsCoeffOfy}
s = y^p - \beta^{p-1}x_1^{p-1}y+ s_0.
\end{equation}
Now we compute $s_0$. Using $\tau_1(s)=s$ and $\tau_2(s)=s$, we get 
\begin{equation*}
\tau_1(s_0) - s_0 = \tau_2(s_0) - s_0 = (\beta^{p-1}-1)x_1^p.
\end{equation*}
Write 
$s_0 = \sum_{i=0}^p s_{0,i} x_2^i$
\;\text{with}\; $s_{0,i } \in \Bbbk[x_1,x_3]$, homogeneous of degree $p-i$. 
Then
\[
\tau_1(s_0) - s_0 
= 
\sum_{i=1}^{p-1} s_{0,i} \sum_{j=0}^{i-1} \binom{i}{j} x_2^j x_1^{i-j}
+s_{0,p } x_1^p
=(\beta^{p-1}-1)x_1^p
\]
from which we see that 
$s_{0,1 } \in \Bbbk\langle x_1^{p-1} \rangle,$ and
$s_{0,i}  = 0 \;\text{for each}\; 2 \leq i \leq p-1$. Hence
\begin{equation}
\label{equation:ExprOfsZero}
s_0 = s_{0,0} + s_{0,1} x_2 + s_{0,p}x_2^p
\end{equation} 
and 
\begin{equation}
\label{equation:sZeroOnesZeropBeta}
s_{0,1 } x_1 + s_{0,p } x_1^p =\tau_1(s_0)-s_0= (\beta^{p-1}-1)x_1^p.
\end{equation}
Similarly
\[
\tau_2(s_0) - s_0 
= 
\sum_{i=1}^{p-1} s_{0,i} \sum_{j=0}^{i-1} \binom{i}{j} x_2^j (\alpha x_1)^{i-j}
+s_{0,p } \alpha^px_1^p
=(\beta^{p-1}-1)x_1^p
\]
from which we see, additionally, that 
\begin{equation}
\label{equation:sZeroOnesZeropAlphaBeta}
s_{0,1 } \alpha x_1 + s_{0,p } \alpha^p x_1^p = (\beta^{p-1}-1)x_1^p.
\end{equation}
Since $\alpha \neq 1$, we get from~\eqref{equation:sZeroOnesZeropBeta}
and~\eqref{equation:sZeroOnesZeropAlphaBeta} that 
\begin{equation*}
\label{equation:sZeroOneInTermsOfAlphasZerop}
s_{0,1 } = -(1 + \alpha + \cdots + \alpha^{p-1})s_{0,p }x_1^{p-1}
\end{equation*}
Therefore
\begin{equation*}
\label{equation:RelationBtwnAlphaBeta}
-(\alpha + \cdots + \alpha^{p-1})s_{0,p} = (\beta^{p-1}-1),
\;\text{i.e.,}\; s_{0,p}=-\lambda.
\end{equation*} Hence
\[
s = 
y^p - \beta^{p-1}x_1^{p-1}y -\lambda x_2^p+ \mu x_1^{p-1}x_2+s_{0,0}=a+s_{0,0}.
\qedhere
\]
\end{proof}

\begin{proposition}
$R$ is a direct summand of $A$.
\end{proposition}

\begin{proof}
It follows from Proposition~\ref{proposition:ApNeqRp}
that $p = \min\{j \mid A_j \neq R_j \}$. 
Note that $\deg \Delta_{A/R } = p(p-1 )$
(Proposition~\ref{proposition:deltaARmainEx}).
Now, apply Proposition~\ref{proposition:eltWithNZTr}
for any $a \in A_p \minus R_p$.
\end{proof}

\begin{proposition}
There does not exist 
$s \in S_1$ 
such that
\[
\frac{\Trace_{F/K}((\Pi_{G'}s)^{p-1})}{\Delta_{A/R}} \in \Bbbk^\times.
\]
\end{proposition}

\begin{proof}
Let $s \in S_1$. 
In view of Observation~\ref{observationbox:PiGprimes}, we may assume that
$y$ appears with a non-zero coefficient in $s$.
We show that the $G'$-orbit of
$s$ has more than $p$ elements. For $0\leq i, j, k<p$, we have
\[
(\tau_1^i\tau_2^j\tau_3^k)(y+\gamma x_2)
=y+\gamma x_2 + (i+j+k\beta + i\gamma + j\gamma\alpha) x_1.
\]
Consider the $\bbF_p$-span of $\{1+\gamma, 1+\gamma\alpha, \beta\}
\subseteq \Bbbk$. Its dimension is at least $1$, since $\beta \neq 0$.
Suppose that its dimension is $1$. If $1+\gamma = 0$, then $1+\gamma\alpha =
1-\alpha$ which is not in the $\bbF_p$-span of $\beta$, so the rank is $2$,
a contradiction. Hence $1+\gamma \neq 0$. Therefore we can write 
$1+\gamma = m\beta$ and $1+\gamma\alpha = n\beta$ with $m, n \in \bbF_p$, 
$m \neq 0$. From this we get $1+(m\beta-1)\alpha = n\beta$, i.e., $1-\alpha = (n-m\alpha)\beta$, i.e., $\beta \in \bbF_p(\alpha)$, a contradiction. 
Hence the $\bbF_p$-span of $\{1+\gamma, 1+\gamma\alpha, \beta\}$ has rank
at least $2$. In other words, the $G'$-orbit of each $s \in S_1$ involving
$y$ has at least $p^2$ elements.
\end{proof}

Since $A$ is a polynomial ring
(\cite[Theorem~5, Appendix]{LandweberStongDepth1987})
and $R$ is a direct summand of $A$, $R$ is a direct summand of $S$.
Further, since $G$ is abelian, it follows that $R$ is a polynomial
ring~\cite[Theorem~1.1]{BroerInvThyAbelianTransvGps2010}.

\begin{examplebox}
Continue with Example~\ref{example:main}.
Let $\widetilde{G }$ be the subgroup of $G$ generated by $\tau_1$, $\tau_2$
and $\sigma$.
Write $B = S^{\widetilde{G}}$.
We show that $B$ is not a direct summand of $S$, though $R = B^{\langle
\tau_3\rangle}$ is a direct summand of $S$.
Write $C = S^{\langle\tau_1, \tau_2\rangle}$.
We now apply the second expression in 
Proposition~\ref{proposition:BetaSigmaGreater} to the
extension $B
\subseteq C$: Note that
\[
\{g \in \langle \tau_1, \tau_2 \rangle \mid gx_i = x_i \;\text{for all}\;
1 \leq i \leq 3 \} = \langle 1 \rangle,
\]
so $\Delta_{C/B} = ((\sigma-1)y)^{p-1} = x_3^{p-1}$. 
Therefore, to conclude that $B$ is not a direct summand of $C$ (and hence
also not of $S$), it suffices, by Proposition~\ref{proposition:eltWithNZTr}, 
to show that $B_1 = C_1$.
One can check this in \cite{M2} and \cite{Singular}, but here is a quick
proof. $C_1$ is the intersection of the eigen-spaces of $\tau_1$ and
$\tau_2$ inside $S_1$ for the eigen-value $1$. Each eigen-space is
three-dimensional, and they are not equal. Hence $C_1$ is two-dimensional, and, therefore $C_1 = \Bbbk \langle x_1, x_3 \rangle \subseteq B_1
\subseteq C_1$. In summary we have the following diagram
\[
\xymatrix{%
A \ar[r]^{\text{split}}  & C \ar[r]^{\text{split}}  & S 
\\
R \ar[r]^{\text{split}}  \ar[u]^{\text{split}} & B \ar[u]_{\text{non-split}} 
}
\]
All the groups in question are transvection groups and each arrow other
than $C \to S$ corresponds to $\ints/p\ints$, generated by the class of a
transvection.
\end{examplebox}

\end{document}